\documentclass[11pt]{amsart}
\usepackage{amsmath}
\usepackage{amsthm}
\usepackage{amssymb}
\usepackage{verbatim}
\usepackage{color}
\newtheorem{theorem}{Theorem}    %[section]
\newtheorem{proposition}[theorem]{Proposition}

\newtheorem{corollary}[theorem]{Corollary}
\newtheorem{lemma}[theorem]{Lemma}

\newtheorem{remark}[theorem]{Remark}
\newtheorem{definition}[theorem]{Definition}
\theoremstyle{definition}

\numberwithin{theorem}{section} \numberwithin{theorem}{section}
\numberwithin{equation}{section}

\def\Rn{{\mathbb{R}^n}}

\allowdisplaybreaks

\begin{document}
\title[Bump conditions and two-weight inequalities for commutators]
{Bump conditions and two-weight inequalities for commutators of fractional integrals}
%title of paper and the running head option

\author{Yongming Wen and Huoxiong Wu$^*$}
%%%%%%%%%%%%%%% footnote %%%%%%%%%%%%%%%%

\subjclass[2010]{%2000 MSC numbers
42B20; 42B25; 42B35; 47B47.
}
%In case \subjclass[2000] command is not effective
%(or the version of amsart.cls is old), write as follows instead:
%\renewcommand{\thefootnote}{\fnsymbol{footnote}}
%\footnote[0]{2000\textit{ Mathematics Subject Classification}.

%Primary 00; Secondary 00.}
%
\keywords{commutators, fractional integrals, two weight inequalities, bump conditions, weighted $BMO(\Rn)$ spaces.}
\thanks{$^*$Corresponding author.}
\thanks{Supported by the NNSF of China (Nos. 11771358, 11871101, 11871254), the President's fund of Minnan Normal University (No. KJ2020020), the scientific research project of The Education Department of Fujian Province (No. JAT200331) and Fujian Key Laboratory of Granular Computing and Applications (Minnan Normal University), China.}%, 11671414).}
%%%%%%%%%%%% Authors addresses %%%%%%%%%%%%%
\address{School of Mathematics and Statistics, Minnan Normal University, Zhangzhou 363000,  China} \email{wenyongmingxmu@163.com}
\address{School of Mathematical Sciences, Xiamen University, Xiamen 361005, China} \email{huoxwu@xmu.edu.cn}

%\address{}
%\email{}

%%%%%%%%%%%%%%%%%%%%%%%%%%%%%%%%%%%%%%%%%

\begin{abstract}
This paper gives new two-weight bump conditions for the sparse operators related to iterated commutators of fractional integrals. As applications, the two-weight bounds for iterated commutators of fractional integrals under more general bump conditions are obtained. Meanwhile, the necessity of two-weight bump conditions as well as the converse of Bloom type estimates for iterated commutators of fractional integrals are also given.
\end{abstract}

\maketitle

\section{Introduction and main results}
Let $0<\alpha<n$, $m\in \mathbb{Z}^+$ and $b\in L_{loc}^{1}(\mathbb{R}^n)$. The fractional integral operator $I_\alpha$ and its higher order commutator $I_{\alpha}^{b,m}$ are defined by
\begin{align*}
I_{\alpha}f(x)=\int_{\mathbb{R}^n}\frac{f(y)}{|x-y|^{n-\alpha}}dy,\quad
I_{\alpha}^{b,m}f(x)=\int_{\mathbb{R}^n}(b(x)-b(y))^m\frac{f(y)}{|x-y|^{n-\alpha}}dy.
\end{align*}
In this paper, we consider two weight estimates for $I_{\alpha}^{b,m}$
\begin{align*}
\Big(\int_{\mathbb{R}^n}|I_{\alpha}^{b,m}f(x)|^q\mu(x)dx\Big)^{1/q}\leq C\Big(\int_{\mathbb{R}^n}|f(x)|^p\nu(x)dx\Big)^{1/p},
\end{align*}
where $(\mu,\nu)$ is a pair of weights. Before this, we recall some backgrounds.

Let $1<p<n/\alpha$ and $1/p-1/q=\alpha/n$. It is well known that $I_\alpha$ is bounded from $L^p(\mathbb{R}^n)$ to $L^q(\mathbb{R}^n)$. Given a function $b\in L_{loc}^{1}(\mathbb{R}^n)$, we say that $b\in BMO(\mathbb{R}^n)$ if
\begin{align*}
\|b\|_{BMO(\mathbb{R}^n)}=\sup_Q\frac{1}{|Q|}\int_Q|b(x)-b_Q|dx<\infty,
\end{align*}
where $b_Q=|Q|^{-1}\int_Qb(x)dx$. In 1982, Chanillo \cite{Chan} proved that if $1<p<n/\alpha$, $1/p-1/q=\alpha/n$ and $b\in BMO(\mathbb{R}^n)$, then $I_{\alpha}^{b,1}$ is bounded from $L^p(\mathbb{R}^n)$ to $L^q(\mathbb{R}^n)$. By a weight $\omega$, we mean a nonnegative locally integrable function on $\mathbb{R}^n$. We say that $\omega\in A_{p,q}$ if
\begin{align*}
[\omega]_{A_{p,q}}=\sup_{Q}\Big(\frac{1}{|Q}\int_Q\omega(x)^qdx\Big)\Big(\frac{1}{|Q|}\int_Q
\omega(x)^{-p'}\Big)^{q/p'}<\infty,~1<p<q<\infty.
\end{align*}
Muckenhoupt and Wheeden \cite{MuW} proved that $I_\alpha$ is bounded from $L^p(\omega^p)$ to $L^q(\omega^q)$, where $0<\alpha<n$, $1<p<n/\alpha$, $1/p-1/q=\alpha/n$ and $\omega\in A_{p,q}$. Under the same conditions as \cite{MuW} with $b\in BMO(\mathbb{R}^n)$, Segovia and Torrea \cite{ST} obtained the weighted $L^p\rightarrow L^q$ boundedness for commutators of fractional integral operators.

Though the forms of two weight inequalities for singular integral operators and related operators are the generalization of one weight inequalities, two weight estimates for operators are more difficult. For instance, it is well known that the $A_p$ condition
\begin{align*}
\sup_Q\Big(\frac{1}{|Q|}\int_Q\omega(x)dx\Big)\Big(\frac{1}{|Q|}\int_Q\omega(x)^{1-p'}dx\Big)^{p-1}
<\infty
\end{align*}
is the sufficient condition for singular integral operators and related operators to be bounded on $L^p(\omega)$. However, in general, the $A_p$ condition for a pair of weights $(\mu,\nu)$
\begin{align}\label{two weight Ap}
\sup_Q\Big(\frac{1}{|Q|}\int_Q\mu(x)dx\Big)\Big(\frac{1}{|Q|}\int_Q\nu(x)^{1-p'}dx\Big)^{p-1}
<\infty
\end{align}
is necessary but never sufficient for operators to be bounded from $L^p(\nu)$ to $L^p(\mu)$, see \cite{CruMP}. To solve this problem, Sawyer \cite{Sawyer} first introduced the following test condition: there is a positive constant $C$ such that for any cubes $Q$
\begin{align*}
\int_QM(\nu^{1-p'}\chi_Q)(x)^p\mu(x)dx\leq C\int_Q\nu(x)^{1-p'}dx<\infty,
\end{align*}
where $M$ is the Hardy-Littlewood maximal operator, he proved that this test condition is necessary and sufficient for $M$ to be bounded from $L^p(\nu)$ to $L^p(\mu)$. However, this condition is very difficult to verify due to the operator $M$ involves in it. This drawback appeals researchers to searching for some simpler sufficient conditions, which are close to \eqref{two weight Ap} in some sense. Neugebaur \cite{Neu} first proved that for some $r>1$, if a pair of weights
$(\mu,\nu)$ satisfies the following power bump condition:
\begin{align*}
\sup_Q\Big(\frac{1}{|Q|}\int_Q\mu(x)^rdx\Big)^{1/r}\Big(\frac{1}{|Q|}\int_Q\nu(x)^{r(1-p')}dx\Big)
^{(p-1)/r}<\infty,
\end{align*}
then
\begin{align*}
\int_{\mathbb{R}^n}(Mf(x))^p\mu(x)dx\leq C\int_{\mathbb{R}^n}|f(x)|^p\nu(x)dx.
\end{align*}

To formulate the following works of seeking for appropriate bump conditions which are sufficient for the two weight inequalities of singular integral operators and related operators. We recall some facts about Orlicz spaces. We say $A(t):[0,\infty)\rightarrow[0,\infty)$ is a Young function if it is increasing, convex, $A(0)=0$ and $A(t)/t\rightarrow\infty$ as $t\rightarrow\infty$. Given a Young function $A$, the associated complementary function $\bar{A}$ is defined by
\begin{align*}
\bar{A}(t)=\sup_{s>0}\{st-A(s)\}.
\end{align*}
Let $1<p<\infty$ and $A$ be a Young function, we say that $A\in B_p$ if
\begin{align*}
\int_{1}^{\infty}\frac{A(t)}{t^p}\frac{dt}{t}<\infty.
\end{align*}
Given a Young function $A$, the Orlicz average on a cube $Q$ of a function $f$ is defined by
\begin{align*}
\|f\|_{A,Q}=\inf\Big\{\lambda>0:\frac{1}{|Q|}\int_QA\Big(\frac{|f(x)|}{\lambda}\Big)dx\leq1\Big\}.
\end{align*}

In 1995, P\'{e}rez \cite{Perez3} improved Neugebaur's result by eliminating the power bump on the left-hand weight $\mu$ and replacing the right-hand weight $\nu$ by the ``Orlicz bump''. Precisely, he proved that if a pair of weights $(\mu,\nu)$ satisfies
\begin{align*}
\sup_Q\|\mu^{1/p}\|_{p,Q}\|\nu^{-1/p}\|_{\Phi,Q}<\infty,~1<p<\infty,
\end{align*}
and $\bar{\Phi}\in B_p$, then $M:L^p(\nu)\rightarrow L^p(\mu)$. While for Calder\'{o}n-Zygmund operator $T$, Cruz-Uribe and P\'{e}rez \cite{CruPe} conjectured that if both terms in \eqref{two weight Ap} were bumped, then $T:L^p(\nu)\rightarrow L^p(\mu)$. This conjecture was partially solved in \cite{NaReTV} and completely solved by Lerner in \cite{Lerner}. Lerner proved that if a pair of weights $(\mu,\nu)$ satisfies
\begin{align}\label{bump conjcture}
\sup_Q\|\mu^{1/p}\|_{\Psi,Q}\|\nu^{-1/p}\|_{\Phi,Q}<\infty,~1<p<\infty,
\end{align}
and $\bar{\Phi}\in B_p, \bar{\Psi}\in B_{p'}$, then $T:L^p(\nu)\rightarrow L^p(\mu)$. The separated bump conjecture, which aries from the work of Cruz-Uribe et al. \cite{CruPe0}, who asserted that $T:L^p(\nu)\rightarrow L^p(\mu)$ provided that \eqref{bump conjcture} is replaced by
\begin{align*}
\sup_Q\|\mu^{1/p}\|_{p,Q}\|\nu^{-1/p}\|_{\Phi,Q}<\infty\quad and\quad \|\mu^{1/p}\|_{\Psi,Q}\|\nu^{-1/p}\|_{p',Q}<\infty.
\end{align*}
In \cite{CruRV}, Cruz-Uribe et al. only proved this conjecture is true for $\Phi(t)=t^{p'}[\log(e+t)]^{p'-1+\delta}$ and $\Psi(t)=t^p[\log(e+t)]^{p-1+\delta}$ for some $\delta>0$. This conjecture is still open, and we refer readers to see \cite{Lacey,LerOR,Li} for more recent works about it. Analogue to the case of singular integral operators, P\'{e}rez \cite{Perez2} gave the following sufficient condition:
\begin{align*}
\sup_Q|Q|^{\alpha/n+1/q-1/p}\|\mu^{1/q}\|_{A,Q}\|\nu^{-1/p}\|_{B,Q}<\infty,~\bar{A}\in B_p, \bar{B}\in B_{q'},
\end{align*}
such that $I_\alpha:L^p(\nu)\rightarrow L^q(\mu)$. The conditions $\bar{A}\in B_p, \bar{B}\in B_{q'}$ were improved to $\bar{A}\in B_{p,q}, \bar{B}\in B_{q',p'}$ in \cite{CruzM}. Here, we say that $A\in B_{p,q}$ if
$$\int_{1}^{\infty}\frac{A(t)^{q/p}}{t^q}\frac{dt}{t}<\infty.$$
Recently, Rahm \cite{R} used ``entropy bumps'' and ``direct comparison bumps'' to get the two weight boundedness for fractional sparse operators.

On the other hand, Cruz-Uribe and Moen \cite{CruzM0} showed that if $b\in BMO(\mathbb{R}^n)$ and a pair of weights $(\mu,\nu)$ satisfies
\begin{align*}
\sup_Q\|\mu^{1/p}\|_{L^p(\log L)^{2p-1+\delta},Q}\|\nu^{-1/p}\|_{L^{p'}(\log L)^{2p'-1+\delta},Q}<\infty,
\end{align*}
then the commutator of Calder\'{o}n-Zygmund operator $T_b$ is bounded from $L^p(\nu)$ to $L^p(\mu)$. This result was recently improved by Lerner et al. \cite{LerOR}, who provided a wider class of weights $(\mu,\nu)$:
\begin{align*}
\sup_Q\|\mu^{1/p}\|_{L^p(\log L)^{(m+1)p-1+\delta}}\|\nu^{-1/p}\|_{B,Q}+\sup_Q
\|\mu^{1/p}\|_{A,Q}\|\nu^{-1/p}\|_{L^p(\log L)^{(m+1)p-1+\delta},Q}<\infty,
\end{align*}
for which $\|T_b^m\|_{L^p(\nu)\rightarrow L^p(\mu)}<\infty$, where $b\in BMO(\mathbb{R}^n)$ and $\bar{A}\in B_{p'},\bar{B}\in B_p$. Very recently, Cruz-Uribe et al. \cite{CruMT} generalized the work in \cite{LerOR} by assuming the Young functions $\bar{A},\bar{C}\in B_{p'}$, $\bar{B},\bar{D}\in B_p$ and $(\mu,\nu)$ satisfies
\begin{align*}
\sup_Q\|\mu^{1/p}\|_{A,Q}\|(b-b_Q)^m\nu^{-1/p}\|_{B,Q}+
\sup_Q\|(b-b_Q)^m\mu^{1/p}\|_{C,Q}\|\nu^{-1/p}\|_{B,Q}<\infty.
\end{align*}
We also refer readers to see \cite{IsPT} for the result in the matrix setting. For the commutators of fractional integral operators, Cruz-Uribe \cite{Cru} showed that if a pair of weights $(\mu,\nu)$ satisfies
\begin{align}\label{improve this condition}
&\sup_Q|Q|^{\alpha/n+1/q-1/p}\|\mu^{1/q}\|_{A,Q}\|\nu^{-1/p}\|_{B,Q}<\infty,
\end{align}
with $A(t)=t^q(\log(e+t))^{2q-1+\delta}$, $B(t)=t^{p'}(\log(e+t))^{2p'-1+\delta}$, then $I_\alpha^{b,1}$ is bounded from $L^p(\nu)$ to $L^p(\mu)$. Recently, Cardenas and Isralowitz \cite{CaIs} established two weight inequality for $I_{\alpha}^{b,1}$ in the matrix setting.

Inspired by the works in \cite{CaIs,CruMT,LerOR}, in this paper, we mainly consider two weight inequalities for $I_{\alpha}^{b,m}$. Our first main result can be formulated as follows.
\begin{theorem}\label{theorem1.1}
Let $1<p\leq q<\infty$, $0<\alpha<n$, $m\in\mathbb{Z}^+$, $b\in L_{loc}^{m}(\mathbb{R}^n)$ and $\mathcal{S}$ be a sparse family.
\begin{itemize}
\item [(1)]Suppose that $A,B,C,D$ are Young functions which satisfy $\bar{A},\bar{C}\in B_{q'}$ and $\bar{B},\bar{D}\in B_{p,q}$. If a pair of weights $(\mu,\nu)$ satisfies
\begin{align*}
&\sup_{Q\in\mathcal{S}}|Q|^{\alpha/n+1/q-1/p}\|\mu^{1/q}
\|_{A,Q}\|(b-b_Q)^m\nu^{-1/p}\|_{B,Q}\\
&\qquad+\sup_{Q\in\mathcal{S}}|Q|^{\alpha/n+1/q-1/p}
\|(b-b_Q)^m\mu^{1/q}\|_{C,Q}\|\nu^{-1/p}\|_{D,Q}<\infty,
\end{align*}
then
\begin{align}\label{new}
\|T_{\alpha}^{\mathcal{S},b,m}f\|_{L^q(\mu)}
+\|(T_{\alpha}^{\mathcal{S},b,m})^{\ast}f\|_{L^q(\mu)}\lesssim \|f\|_{L^p(\nu)}.
\end{align}
\item [(2)]Conversely, if \eqref{new} holds, then
\begin{align*}
&\sup_{Q\in\mathcal{S}}|Q|^{{\alpha}/{n}+{1}/{q}-{1}/{p}}
\|\mu^{1/q}\|_{q,Q}\|(b-b_Q)^m\nu^{-1/p}\|_{p',Q}\\
&\qquad+\sup_{Q\in\mathcal{S}}|Q|^{{\alpha}/{n}+{1}/{q}-{1}/{p}}
\|(b-b_Q)^m\mu^{1/q}\|_{q,Q}\|\nu^{-1/p}\|_{p',Q}<\infty.
\end{align*}
\end{itemize}
Here
\begin{align*}
T_{\mathcal{S},\alpha}^{b,m}f(x)=\sup_{Q\in\mathcal{S}}|Q|^{\alpha/n}\Big(|Q|^{-1}\int_Q
|b(x)-b_Q|^m|f(x)|dx\Big)\chi_Q(x),
\end{align*}
and
\begin{align*}
(T_{\mathcal{S},\alpha}^{b,m})^{\ast}f(x)=\sup_{Q\in\mathcal{S}}|Q|^{\alpha/n}|b(x)-b_Q|^m
\Big(|Q|^{-1}\int_Q|f(x)|dx\Big)\chi_Q(x).
\end{align*}
\end{theorem}

As an application, we can obtain the following two weight bump conditions for iterated commutators $I_\alpha^{b, m}$.

\begin{theorem}\label{theorem1.2}
Let $1<p\leq q<\infty$, $0<\alpha<n$, $m\in\mathbb{Z}^+$, $b\in L_{loc}^{m}(\mathbb{R}^n)$ and $I_\alpha^{b,m}$ be commutators of fractional integral operators. Suppose that $A,B,C,D$ are Young functions which satisfy $\bar{A},\bar{C}\in B_{q'}$ and $\bar{B},\bar{D}\in B_{p,q}$. If a pair of weights $(\mu,\nu)$ satisfy
\begin{align*}
&\sup_{Q\in\mathcal{S}}|Q|^{\alpha/n+1/q-1/p}\|\mu^{1/q}
\|_{A,Q}\|(b-b_Q)^m\nu^{-1/p}\|_{B,Q}\\
&\qquad+\sup_{Q\in\mathcal{S}}|Q|^{\alpha/n+1/q-1/p}
\|(b-b_Q)^m\mu^{1/q}\|_{C,Q}\|\nu^{-1/p}\|_{D,Q}<\infty,
\end{align*}
then $\|I_{\alpha}^{b,m}f\|_{L^q(\mu)}\lesssim\|f\|_{L^p(\nu)}$.
\end{theorem}

Furthermore, as a consequence of Theorem \ref{theorem1.2}, we can obtain the more traditional bump conditions by assuming that the multiplier $b$ lies in an oscillation class related to $BMO(\mathbb{R}^n)$.

\begin{theorem}\label{theorem1.3}
Let $1<p\leq q<\infty$, $0<\alpha<n$, $m\in\mathbb{Z}^+$ and $I_\alpha^{b,m}$ be commutators of fractional integral operators. Assume that $A,B,C,D,X,Y$ are Young functions which satisfy $\bar{A},\bar{C}\in B_{q'}$, $\bar{B},\bar{D}\in B_{p,q}$ and $X,Y$ satisfy
\begin{align*}
X^{-1}(t)\lesssim\frac{B^{-1}(t)}{\Phi^{-1}(t)^m}~and~Y^{-1}(t)\lesssim
\frac{C^{-1}(t)}{\Phi^{-1}(t)^m}
\end{align*}
for large $t$. If $b\in Osc(\Phi)$ and a pair of weights $(\mu,\nu)$ satisfies
\begin{align*}
&\sup_{Q}|Q|^{\alpha/n+1/q-1/p}\|\mu^{1/q}
\|_{A,Q}\|\nu^{-1/p}\|_{X,Q}\\
&\qquad+\sup_{Q}|Q|^{\alpha/n+1/q-1/p}
\|\mu^{1/q}\|_{Y,Q}\|\nu^{-1/p}\|_{D,Q}<\infty,
\end{align*}
then $\|I_{\alpha}^{b,m}f\|_{L^q(\mu)}\lesssim\|b\|_{Osc(\Phi)}^m\|f\|_{L^p(\nu)}$, where $Osc(\Phi)$ is the space of functions $b\in L_{loc}^{1}(\mathbb{R}^n)$ with
$$\|b\|_{Osc(\Phi)}=\sup_Q\|b-b_Q\|_{\Phi,Q}<\infty.$$
\end{theorem}

When $b\in BMO(\mathbb{R}^n)$, we may take $\Phi(t)=\exp t-1$ in Theorem \ref{theorem1.3}. Then we have the following result.

\begin{corollary}\label{corollary1.4}
Let $1<p\leq q<\infty$, $0<\alpha<n$, $m\in\mathbb{Z}^+$ and $I_\alpha^{b,m}$ be commutators of fractional integral operators. Suppose that $A,D$ are Young functions which satisfy $\bar{A}\in B_{q'}$, $\bar{D}\in B_{p,q}$. If $b\in BMO(\mathbb{R}^n)$ and a pair of weights $(\mu,\nu)$ satisfy
\begin{align*}
&\sup_{Q}|Q|^{\alpha/n+1/q-1/p}\|\mu^{1/q}
\|_{A,Q}\|\nu^{-1/p}\|_{L^{p'}(\log L)^{(m+1)p'-1+\delta},Q}\\
&\qquad+\sup_{Q}|Q|^{\alpha/n+1/q-1/p}
\|\mu^{1/q}\|_{L^{q}(\log L)^{(m+1)q-1+\delta},Q}\|\nu^{-1/p}\|_{D,Q}<\infty,
\end{align*}
then $\|I_{\alpha}^{b,m}f\|_{L^q(\mu)}\lesssim\|b\|_{BMO(\mathbb{R}^n)}^m\|f\|_{L^p(\nu)}$.
\end{corollary}

In particular, if we take $A(t)=t^q[\log(e+t)]^{q-1+\delta}$ and $D(t)=t^{p'}[\log(e+t)]^{p'-1+\delta}$, then we can get the
following more general two-weight bump conditions than (\ref{improve this condition}) for $I_\alpha^{b,m}$.

\begin{corollary}\label{corollary1.5}
Let $1<p\leq q<\infty$, $0<\alpha<n$, $m\in\mathbb{Z}^+$ and $I_\alpha^{b,m}$ be commutators of fractional integral operators. If $b\in BMO(\mathbb{R}^n)$ and a pair of weights $(\mu,\nu)$ satisfies
\begin{align}\label{weaker}
&\sup_{Q}|Q|^{\alpha/n+1/q-1/p}\|\mu^{1/q}
\|_{L^{q}(\log L)^{q-1+\delta},Q}\|\nu^{-1/p}\|_{L^{p'}(\log L)^{(m+1)p'-1+\delta},Q}\\
&\qquad+\sup_{Q}|Q|^{\alpha/n+1/q-1/p}
\|\mu^{1/q}\|_{L^{q}(\log L)^{(m+1)q-1+\delta},Q}\|\nu^{-1/p}\|_{L^{p'}(\log L)^{p'-1+\delta},Q}<\infty\nonumber
\end{align}
for some $\delta>0$, then $\|I_{\alpha}^{b,m}f\|_{L^q(\mu)}\lesssim\|b\|_{BMO(\mathbb{R}^n)}^m\|f\|_{L^p(\nu)}$.
\end{corollary}

\begin{remark}It is clear that the bump condition in $(1.5)$ for $m=1$ is more general than one in $(1.3)$. Therefore, Corollary \ref{corollary1.5} is an essential improvement and extension to the corresponding result in \cite{Cru}.
\end{remark}

Next, we turn to the necessity of bump conditions for the two-weight boundedness of $I_{\alpha}^{b,m}$, which is addressed by the following theorem.

\begin{theorem}\label{theorem1.7}
Let $1<p\leq q<\infty$, $0<\alpha<n$, $m\in\mathbb{Z}^+$ and $I_\alpha^{b,m}$ be commutators of fractional integral operators. Suppose that $\mu$ is a doubling weight and for any $b\in BMO(\mathbb{R}^n)$,
$$\|I_{\alpha}^{b,m}f\|_{L^{q,\infty}(\mu)}\lesssim\|b\|_{BMO(\mathbb{R}^n)}^m
\|f\|_{L^p(\nu)}.$$
Then
\begin{align*}
\sup_Q|Q|^{\frac{\alpha}{n}+\frac{1}{q}-\frac{1}{p}}\|\mu^{1/q}\|_{L^q,Q}
\|\nu^{-1/p}\|_{L^{p'}(\log L)^{mp'},Q}<\infty.
\end{align*}
\end{theorem}

Finally, we consider the inverse result related to Bloom type estimate for $I_\alpha^{b,m}$. We first recall the relevant definition and backgrounds.
Let $\eta$ be a weight, we say that $b\in BMO_\eta$ if
$$\|b\|_{BMO_\eta}:=\sup_Q\frac{1}{\eta(Q)}\int_Q|b(x)-b_Q|dx<\infty.$$
Bloom \cite{Bl0} first charactered $BMO_\eta$ via the two weight estimate of commutator of Hilbert transform $H$. For the commutator of fractional integral operator, Accomazzo et al. \cite{AMR} proved that if $\lambda,\mu\in A_{p,q}$ and
$\eta=\big({\mu}{\lambda^{-1}}\big)^{1/m}$, then
$$b\in BMO_\eta\Rightarrow\|I_{\alpha}^{b,m}\|_{L^q(\lambda^q)}\lesssim\|b\|_{BMO_\eta}^m
\|f\|_{L^p(\mu^p)}$$
and
$$\|I_{\alpha}^{b,m}\|_{L^q(\lambda^q)}\lesssim\|f\|_{L^p(\mu^p)}\Rightarrow b\in BMO_\eta.$$
The corresponding result for $m=1$ was obtained by Holmes et al. in \cite{HRS}. Our next theorem can be regarded as  the converse of the above Bloom type estimate for $I_\alpha^{b,m}$.

\begin{theorem}\label{theorem1.8}
Let $0<\alpha<n$, $1<p<n/\alpha$, $1/p-1/q=\alpha/n$, $m\in\mathbb{Z}^+$, $\lambda,\mu\in A_{p,q}$ and $I_\alpha^{b,m}$ be commutator of fractional integral operators. If $\eta$ is an arbitrary weight which satisfies
\begin{align}\label{boundedness}
b\in BMO_\eta\Rightarrow\|I_{\alpha}^{b,m}\|_{L^q(\lambda^q)}\lesssim\|b\|_{BMO_\eta}^m
\|f\|_{L^p(\mu^p)}
\end{align}
and
\begin{align}\label{necessity}
\|I_{\alpha}^{b,m}\|_{L^q(\lambda^q)}\lesssim\|f\|_{L^p(\mu^p)}\Rightarrow b\in BMO_\eta,
\end{align}
then $\eta\sim\big({\mu}{\lambda^{-1}}\big)^{1/m}$ almost where.
\end{theorem}

We organize the rest of the paper as follows. Section 2 is devoted to the proofs of Theorems \ref{theorem1.1}-\ref{theorem1.3}, Corollarys \ref{corollary1.4} and \ref{corollary1.5}. In Section 3, we will show Theorem \ref{theorem1.7} and Theorem \ref{theorem1.8} will be given in Section 4.

We end this section by making some conventions. We denote $f\lesssim g$, $f\thicksim g$ if $f\leq Cg$ and $f\lesssim g \lesssim f$, respectively. For any ball $B:=B(x_0,r)\subset \mathbb{R}^n$, $x_0$ and $r$ are the center and the radius of $B$, respectively, and $f_B$ means the mean value of $f$ over $B$, $\chi_B$ represents the characteristic function of $B$. For any cube $Q\subset\mathbb{R}^n$, the diameter of $Q$ is denoted by diam $Q$. $C_{c}^{\infty}(\Rn)$ is the space of all smooth functions with compact support.

\section{Two-weight boundedness for $I_\alpha^{b,m}$}

In this section, we will prove Theorems \ref{theorem1.1}-\ref{theorem1.3} and Corollary \ref{corollary1.4}-\ref{corollary1.5}. To begin with recalling some notation,definitions and facts related to sparse families (see \cite{LerNa,Perey} for more details).
Given a cube $Q\subset\mathbb{R}^n$, let $\mathcal{D}(Q)$ be the set of cubes obtained by repeatedly subdividing $Q$ and its descendants into $2^n$ congruent subcubes.
\begin{definition}
A collection of cubes $\mathcal{D}$ is called a dyadic lattice if it satisfies the following properties:\\
$(1)$ if $Q\in\mathcal{D}$, then every child of $Q$ is also in $\mathcal{D}$;\\
$(2)$ for every two cubes $Q_1, Q_2\in\mathcal{D}$, there is a common ancestor $Q\in\mathcal{D}$ such that $Q_1, Q_2\in\mathcal{D}(Q)$;\\
$(3)$ for any compact set $K\subset\mathbb{R}^n$, there is a cube $Q\in\mathcal{D}$ such that $K\subset Q$.
\end{definition}

\begin{definition}
A subset $\mathcal{S}\subset\mathcal{D}$ is called an $\eta$-sparse family with $\eta\in(0,1)$ if for every cube $Q\in\mathcal{S}$, there is a measurable subset $E_Q\subset Q$ such that $\eta|Q|\leq|E_Q|$, and the sets $\{E_Q\}_{Q\in\mathcal{S}}$ are mutually disjoint.
\end{definition}

In \cite{AMR}, Accomazzo et al. proved the following sparse dominations for commutators of fractional integral operators.
\begin{lemma}{\rm(cf. \cite{AMR})}\label{sparse domination}
Let $0<\alpha<n$ and $m\in\mathbb{Z}^+$. For every $f\in C_c^\infty(\mathbb{R}^n)$ and $b\in L_{loc}^{m}(\mathbb{R}^n)$, there exist a family $\{\mathcal{D}_j\}_{j=1}^{3^n}$ of dyadic lattices and a family $\{\mathcal{S}_j\}_{j=1}^{3^n}$ of sparse families such that $\mathcal{S}_j\subset\mathcal{D}_j$, for each $j$, and
$$|I_{\alpha}^{b,m}f(x)|\lesssim\sum_{j=1}^{3^n}\sum_{Q\in\mathcal{S}_j}\sum_{k=0}^{m}
|b(x)-b_Q|^{m-k}|Q|^{\alpha/n}\Big(\frac{1}{|Q|}\int_Q|b(x)-b_Q|^k|f(x)|dx\Big)\chi_Q(x).$$
\end{lemma}

Based on Lemma \ref{sparse domination}, we can prove the following lemma.
\begin{lemma}\label{lemma2.4}
Let $0<\alpha<n$, $m\in\mathbb{Z}^+$, $b\in L_{loc}^{m}(\mathbb{R}^n)$ and $I_\alpha^{b,m}$ be commutators of fractional integral operators. Then for $f\in C_c^\infty(\mathbb{R}^n)$, there exist $3^n$ sparse families $\mathcal{S}_j\subset\mathcal{D}_j$, $j=1,\cdots,3^n$, such that
\begin{align*}
|I_{\alpha}^{b,m}f(x)|\lesssim\sum_{j=1}^{3^n}(T_{\mathcal{S}_j,\alpha}^{b,m}f(x)+
(T_{\mathcal{S}_j,\alpha}^{b,m})^\ast f(x)),
\end{align*}
where $T_{\mathcal{S},\alpha}^{b,m}$ and $(T_{\mathcal{S},\alpha}^{b,m})^\ast$ are defined in Theorem \ref{theorem1.1}.
\end{lemma}
\begin{proof}
Fix a sparse family $\mathcal{S}$, let $Q\in\mathcal{S}$ and $x\in Q$, then
\begin{align*}
&|Q|^{\alpha/n}\sum_{k=0}^m|b(x)-b_Q|^{m-k}\frac{1}{|Q|}\int_Q|b(y)-b_Q|^k|f(y)|dy\\
&\quad\leq\frac{1}{|Q|}\int_Q\Big(\sum_{k=0}^m\max\{|b(x)-b_Q|,|b(y)-b_Q|\}^m\Big)
|f(y)|dy|Q|^{\alpha/n}\\
&\quad=(m+1)\frac{1}{|Q|}\int_Q\max\{|b(x)-b_Q|^m,|b(y)-b_Q|^m\}|f(y)|dy|Q|^{\alpha/n}\\
&\quad\lesssim|b(x)-b_Q|^m\frac{1}{|Q|}\int_Q|f(y)|dy|Q|^{\alpha/n}+
\frac{1}{|Q|}\int_Q|b(y)-b_Q|^m|f(y)|dy|Q|^{\alpha/n}.
\end{align*}
This, together with Lemma \ref{sparse domination}, leads to the desired conclusion and completes the proof of Lemma \ref{lemma2.4}.
\end{proof}

The proof of Theorem \ref{theorem1.1} is converted into the following two propositions.

\begin{proposition}\label{pro2.5}
Let $0<\alpha<n, m\in\mathbb{Z}^+$, $b\in L_{loc}^{m}(\mathbb{R}^n)$ and $\mathcal{S}$ be a sparse family. Assume that $1<p\leq q<\infty$ and $A,B$ are Young functions that satisfy $\bar{A}\in B_{q'},\bar{B}\in B_{p,q}$. If $(\mu,\nu)$ is a pair of weights that satisfies
\begin{align*}
\sup_{Q\in\mathcal{S}}|Q|^{{\alpha}/{n}+{1}/{q}-{1}/{p}}
\|\mu^{1/q}\|_{A,Q}\|(b-b_Q)^m\nu^{-1/p}\|_{B,Q}<\infty,
\end{align*}
then
\begin{align}\label{2.1}
\|T_{\mathcal{S},\alpha}^{b,m}f\|_{L^q(\mu)}\leq C\|f\|_{L^p(\nu)}.
\end{align}

Conversely, if $T_{\mathcal{S},\alpha}^{b,m}$ satisfies \eqref{2.1}, then the pair of weights $(\mu,\nu)$ satisfies
\begin{align*}
\sup_{Q\in\mathcal{S}}|Q|^{{\alpha}/{n}+{1}/{q}-{1}/{p}}
\|\mu^{1/q}\|_{q,Q}\|(b-b_Q)^m\nu^{-1/p}\|_{p',Q}<\infty.
\end{align*}
\end{proposition}

\begin{proof}
By duality, there exists nonnegative measurable function $g\in L^{q'}(\mu)$ with $\|g\|_{L^{q'}(\mu)}=1$ such that
\begin{align}\label{2.2}
\|T_{\mathcal{S},\alpha}^{b,m}\|_{L^q(\mu)}&=\int_{\mathbb{R}^n}T_{\mathcal{S},\alpha}
^{b,m}f(x)g(x)\mu(x)dx\\
&\leq\sum_{Q\in\mathcal{S}}|Q|^{\alpha/n+1}\Big(\frac{1}{|Q|}\int_Q|b(x)-b_Q|^m|f(x)|dx\Big)
\Big(\frac{1}{|Q|}\int_Q|g(x)|\mu(x)dx\Big)\nonumber.
\end{align}
Let $1/p-1/q=\beta/n$, it was proved in \cite{CruzM} that
\begin{align*}
M_{\beta,\bar{B}}: L^p(\mathbb{R}^n)\rightarrow L^q(\mathbb{R}^n).
\end{align*}
From this, by \eqref{2.2}, the generalized H\"{o}lder inequality and our assumptions yield that
\begin{align*}
\|T_{\mathcal{S},\alpha}^{b,m}\|_{L^q(\mu)}&\leq\sum_{Q\in\mathcal{S}}
\|(b-b_Q)^m\nu^{-1/p}\|_{B,Q}\|f\nu^{1/p}\|_{\bar{B},Q}\|\mu^{1/q}\|_{A,Q}
\|g\mu^{1/q'}\|_{\bar{A},Q}|Q|^{1+\frac{\alpha}{n}+\frac{1}{q}-\frac{1}{p}+\frac{\beta}{n}}\\
&\lesssim|E_Q||Q|^{\beta/n}\|f\nu^{1/p}\|_{\bar{B},Q}\|g\mu^{1/q'}\|_{\bar{A},Q}\\
&\leq\int_{\mathbb{R}^n}M_{\bar{A}}(g\mu^{1/q'})(x)M_{\beta,\bar{B}}(f\nu^{1/p})(x)dx\\
&\leq\|M_{\beta,\bar{B}}(f\nu^{1/p})\|_{L^q}\|M_{\bar{A}}(g\mu^{1/q'})\|_{L^{q'}}
\lesssim\|f\|_{L^p(\nu)}.
\end{align*}

Next, we turn to prove necessity. Fix $Q\in\mathcal{S}$, let $f=|b-b_Q|^{m(p'-1)}\nu^{-p'/p}\chi_Q$. For $x\in Q$, it is easy to see that
\begin{align*}
T_{\mathcal{S},\alpha}^{b,m}f(x)\geq
|Q|^{\alpha/n-1}\int_Q|b(x)-b_Q|^{mp'}\nu(x)^{-p'/p}dx,
\end{align*}
which implies that
\begin{align*}
\Big(\int_QT_{\mathcal{S},\alpha}^{b,m}f(x)^q\mu(x)dx\Big)^{1/q}\geq
|Q|^{\alpha/n-1}\int_Q|b(x)-b_Q|^{mp'}\nu(x)^{-p'/p}dx\Big(\int_Q\mu(x)dx\Big)^{1/q}.
\end{align*}
On the other hand,
\begin{align*}
\Big(\int_QT_{\mathcal{S},\alpha}^{b,m}f(x)^q\mu(x)dx\Big)^{1/q}&\leq C\Big(\int_{\mathbb{R}^n}|f(x)|^p\nu(x)dx\Big)^{1/p}\\
&=C\Big(\int_Q|b(x)-b_Q|^{mp'}\nu(x)^{-p'/p}dx\Big)^{1/p}.
\end{align*}
Hence, we conclude that
\begin{align*}
&|Q|^{\alpha/n-1}\int_Q|b(x)-b_Q|^{mp'}\nu(x)^{-p'/p}dx\Big(\int_Q\mu(x)dx\Big)^{1/q}\\
&\quad\leq C\Big(\int_Q|b(x)-b_Q|^{mp'}\nu(x)^{-p'/p}dx\Big)^{1/p}.
\end{align*}
The desired result follows by rearranging the above terms.
\end{proof}

Similarly, we can obtain the following proposition, and we leave the details for the interested readers.

\begin{proposition}\label{pro2.6}
Let $0<\alpha<n, m\in\mathbb{Z}^+$, $b\in L_{loc}^{m}(\mathbb{R}^n)$ and $\mathcal{S}$ be a sparse family. Assume that $1<p\leq q<\infty$ and $C,D$ are Young functions that satisfy $\bar{C}\in B_{q'},\bar{D}\in B_{p,q}$. If $(\mu,\nu)$ is a pair of weights that satisfies
\begin{align*}
\sup_{Q\in\mathcal{S}}|Q|^{{\alpha}/{n}+{1}/{q}-{1}/{p}}
\|(b-b_Q)^m\mu^{1/q}\|_{C,Q}\|\nu^{-1/p}\|_{D,Q}<\infty,
\end{align*}
then
\begin{align}\label{2.3}
\|(T_{\mathcal{S},\alpha}^{b,m})^\ast f\|_{L^q(\mu)}\leq C\|f\|_{L^p(\nu)}.
\end{align}

Conversely, if $(T_{\mathcal{S},\alpha}^{b,m})^\ast$ satisfies \eqref{2.3}, then the pair of weights $(\mu,\nu)$ satisfies
\begin{align*}
\sup_{Q\in\mathcal{S}}|Q|^{{\alpha}/{n}+{1}/{q}-{1}/{p}}
\|(b-b_Q)^m\mu^{1/q}\|_{q,Q}\|\nu^{-1/p}\|_{p',Q}<\infty.
\end{align*}
\end{proposition}

\begin{proof}[Proofs of Theorems \ref{theorem1.1} and \ref{theorem1.2}]
Theorem \ref{theorem1.1} follows from Propositions \ref{pro2.5} and \ref{pro2.6}, and Theorem \ref{theorem1.2} follows from Lemma \ref{lemma2.4} and Theorem \ref{theorem1.1}.
\end{proof}

Next, we prove Theorem \ref{theorem1.3}. We first recall the following lemma.

\begin{lemma}{\rm(cf. \cite{CruMP})}\label{general Holder}
Let $A,B$ be continuous and strictly increasing functions on $[0,\infty)$ and $C$ be Young function that satisfies $A^{-1}(t)B^{-1}(t)\lesssim C^{-1}(t)$ for $t$ large. Then
$$\|fg\|_{C,Q}\lesssim\|f\|_{A,Q}\|g\|_{B,Q}.$$
\end{lemma}

\begin{proof}[Proof of Theorem \ref{theorem1.3}]
Denote $\Phi_m(t)=\Phi(t^{1/m})$. Since $B,X,\Phi$ satisfy
$$\Phi^{-1}(t)^mX^{-1}(t)\lesssim B^{-1}(t)$$ for $t$ large, by Lemma \ref{general Holder}, we have that
\begin{align*}
\|(b-b_Q)^m\nu^{-1/p}\|_{B,Q}&\lesssim\|(b-b_Q)^m\|_{\Phi_m,Q}\|\nu^{-1/p}\|_{X,Q}\\
&=\|(b-b_Q)\|_{\Phi,Q}^m\|\nu^{-1/p}\|_{X,Q}.
\end{align*}
Therefore,
\begin{align*}
&\sup_Q|Q|^{{\alpha}/{n}+{1}/{q}-{1}/{p}}\|\mu^{1/q}\|_{A,Q}
\|(b-b_Q)^m\nu^{-1/p}\|_{B,Q}\\
&\quad\lesssim\|b\|_{Osc(\Phi)}^m\sup_Q|Q|^{{\alpha}/{n}+{1}/{q}-{1}/{p}}
\|\mu^{1/q}\|_{A,Q}\|\nu^{-1/p}\|_{X,Q}<\infty.
\end{align*}
By Proposition \ref{pro2.5}, we get that $$\|T_{\mathcal{S},\alpha}^{b,m} f\|_{L^q(\mu)}\leq C\|f\|_{L^p(\nu)}.$$ Similarly, we have
\begin{align*}
&\sup_Q|Q|^{{\alpha}/{n}+{1}/{q}-{1}/{p}}\|(b-b_Q)^m\mu^{1/q}
\|_{C,Q}\|\nu^{-1/p}\|_{D,Q}\\
&\quad\lesssim\|b\|_{Osc(\Phi)}^m\sup_Q|Q|^{{\alpha}/{n}+{1}/{q}-{1}/{p}}
\|\mu^{1/q}\|_{Y,Q}\|\nu^{-1/p}\|_{D,Q}<\infty.
\end{align*}
This, together with Proposition \ref{pro2.6}, deduces that $$\|(T_{\mathcal{S},\alpha}^{b,m})^\ast f\|_{L^q(\mu)}\leq C\|f\|_{L^p(\nu)}.$$ Summing up the above estimates with Lemma \ref{lemma2.4}, we obtain the conclusion of Theorem \ref{theorem1.3}.
\end{proof}

To prove Corollary \ref{corollary1.4}, we need to recall the following fact. For $\varphi(t)=t^p(\log(e+t))^q$ with $p>1$ and $q\in\mathbb{R}$, Cruz-Uribe et al. \cite{CruPe00} showed that
\begin{align}\label{Young function}
\varphi^{-1}(t)\sim\frac{t^{1/p}}{(\log(e+t))^{q/p}},\quad \bar{\varphi}(t)=\frac{t^{p'}}
{(\log(e+t))^{p'q/p}}.
\end{align}
Now, we give the proof of Corollary \ref{corollary1.4}.

\begin{proof}[Proof of Corollary \ref{corollary1.4}]
To prove this corollary, we need only to choose some Young functions that satisfy the conditions of Theorem \ref{theorem1.3}. For some $\delta>0$, choose
$$X(t)=t^{p'}[\log(e+t)]^{(m+1)p'-1+\delta},~ Y(t)=t^q[\log(e+t)]^{(m+1)q-1+\delta},$$
$$B(t)=t^{p'}[\log(e+t)]^{p'-1+\delta},~C(t)=t^q[\log(e+t)]^{q-1+\delta},~\Phi(t)=e^t-1.$$
It is not hard to check that
$$\bar{B}(t)\sim\frac{t^{p}}{[\log(e+t)]^{1+{p\delta}/{p'}}}\in B_{p}\subset B_{p,q},~ \bar{C}(t)\sim\frac{t^{q'}}{[\log(e+t)]^{1+{q'\delta}/{q}}}\in B_{q'}$$
and $\Phi^{-1}(t)=\log(e+t)$. By \eqref{Young function}, we have that
$$X^{-1}(t)\sim\frac{t^{1/p'}}{[\log(e+t)]^{m+1/p+\delta/p'}},~
Y^{-1}(t)\sim\frac{t^{1/q}}{[\log(e+t)]^{m+1/q'+\delta/q}},$$
$$B^{-1}(t)\sim\frac{t^{1/p'}}{[\log(e+t)]^{1/p+\delta/p'}},~
C^{-1}(t)\sim\frac{t^{1/q}}{[\log(e+t)]^{1/q'+\delta/q}}.$$
Then
$$X^{-1}(t)\Phi^{-1}(t)^m\sim\frac{t^{1/p'}}{[\log(e+t)]^{m+1/p+\delta/p'}}[\log(e+t)]^m\sim B^{-1}(t),$$
$$Y^{-1}(t)\Phi^{-1}(t)^m\sim\frac{t^{1/q}}{[\log(e+t)]^{m+1/q'+\delta/q}}[\log(e+t)]^m\sim C^{-1}(t).$$
Finally, by the John-Nirenberg inequality and $t\lesssim \Phi(t)$, we get $\|b\|_{BMO(\mathbb{R}^n)}\sim\|b\|_{Osc(\Phi)}$. Thus, Theorem \ref{theorem1.3} implies Corollary \ref{corollary1.4}.
\end{proof}

\begin{proof}[Proof of Corollary \ref{corollary1.5}]
Choosing $A(t)=t^q[\log(e+t)]^{q-1+\delta}$, $D(t)=t^{p'}[\log(e+t)]^{p'-1+\delta}$ in Corollary \ref{corollary1.4}. Then
$$\bar{A}(t)\sim\frac{t^{q'}}{[\log(e+t)]^{1+{q'\delta}/{q}}}\in B_{q'},~\bar{D}(t)\sim\frac{t^p}{[\log(e+t)]^{1+{p\delta}/{p'}}}\in B_{p}\subset B_{p,q}.$$
Hence, Corollary \ref{corollary1.5} directly follows from Corollary \ref{corollary1.4}.
\end{proof}

\section{Necessity of two weight inequalities for $I_\alpha^{b,m}$}
In this section, we give the proof of Theorem \ref{theorem1.7}. To prove Theorem \ref{theorem1.7}, we need the following two lemmas.

\begin{lemma}\label{lm3.1}
Let $K_\alpha(x,y)=\frac{1}{|x-y|^{n-\alpha}}$. Then for each $A\geq4$ and each ball $B:=B(y_0,r)$, there exists a disjoint ball $\tilde{B}:=B(x_0,r)$ with dist$(B,\tilde{B})\sim Ar$ satisfies $|K_\alpha(x_0,y_0)|=\frac{1}{A^{n-\alpha}r^{n-\alpha}}$, and for any $y\in B$ and $x\in\tilde{B}$, there holds
\begin{align*}
|K_\alpha(x,y)-K_\alpha(x_0,y_0)|\lesssim\frac{\epsilon_A}{A^{n-\alpha}r^{n-\alpha}},
\end{align*}
where $\epsilon_A\rightarrow0$ as $A\rightarrow\infty$.
\end{lemma}
\begin{proof}
Fix a ball $B=B(y_0,r)$ and $A\geq4$, take $x_0=y_0+Ar\theta_0$, where $\theta_0\in\mathbb{S}^{n-1}$. Let $\tilde{B}:=B(x_0,r)$, it is easy to see that dist$(B,\tilde{B})\sim Ar$ and $K_\alpha(x_0,y_0)=\frac{1}{|x_0-y_0|^{n-\alpha}}=\frac{1}{A^{n-\alpha}r^{n-\alpha}}$. For any $y\in B$ and $x\in\tilde{B}$, by the mean value theorem, we have
\begin{align*}
|K_\alpha(x,y)-K_\alpha(x_0,y_0)|&\leq|K_\alpha(x,y)-K_\alpha(x_0,y)|+|K_\alpha(x_0,y)-K_\alpha(x_0,y_0)|\\
&\lesssim\frac{|x-x_0|}{|x_0-y|^{n-\alpha+1}}\leq\frac{1/A}{(Ar)^{n-\alpha}}
=:\frac{\epsilon_A}{(Ar)^{n-\alpha}}.
\end{align*}
\end{proof}

\begin{lemma}{\rm(cf. \cite{LerOR})}\label{lm3.2}
Assume that $f\in BMO(\mathbb{R}^n)$, and let $Q$ be a cube such that $f_Q=0$. Then there exists a function $\varphi$ such that $\varphi=f$ on $Q$, $\varphi=0$ on $\mathbb{R}^n\backslash2Q$ and $\|\varphi\|_{BMO(\mathbb{R}^n)}\lesssim\|f\|_{BMO(\mathbb{R}^n)}$.
\end{lemma}

\begin{proof}[Proof of Theorem \ref{theorem1.7}]
For any cube $Q\subset\mathbb{R}^n$, we define
\begin{align*}
g(x)=\log^+\Big(\frac{M(\nu^{1-p'}\chi_Q)(x)}{(\nu^{1-p'})_Q}\Big).
\end{align*}
It is well known that $g\in BMO(\mathbb{R}^n)$. Moreover, the Kolmogogov inquality yields that
\begin{align*}
\int_Q(M(f\chi_Q))^\delta\lesssim\Big(\frac{1}{|Q|}\int_Q|f|\Big)^\delta|Q|,~0<\delta<1.
\end{align*}
We then have $g_Q\lesssim1$. According to Lemma \ref{lm3.2}, there is a function $\varphi$ satisfying $\varphi=g-g_Q$ on $Q$, $\varphi=0$ outside $2Q$ and $\|\varphi\|_{BMO(\mathbb{R}^n)}\lesssim1$. Choosing a ball $B$ such that the centre of $B$ is the same as cube $Q$ and $r=diam~ Q$. Then by Lemma \ref{lm3.1}, there is a ball $\tilde{B}$ such that the centre of $\tilde{B}$ is the same as cube $B$ and dist$(B,\tilde{B})=Ar$, where $A\geq4$ will be determined later.

Now, we return to prove our theorem. By duality, we find that the condition
\begin{align*}
\|I_\alpha^{b,m}f\|_{L^{q,\infty}(\mu)}\lesssim\|b\|_{BMO(\mathbb{R}^n)}^m\|f\|_{L^p(\nu)}
\end{align*}
is equivalent to the condition
\begin{align}\label{duality}
\|(I_\alpha^{b,m})^\ast f\|_{L^{p'}(\nu^{1-p'})}\lesssim\|b\|_{BMO(\mathbb{R}^n)}^m\|f/\mu\|_{L^{q',1}(\mu)}.
\end{align}
One can check that $(I_\alpha^{b,m})^\ast=(-1)^mI_\alpha^{b,m}$, hence, we can still deal with \eqref{duality} by considering $I_\alpha^{b,m}$. Let $b=\varphi$, then for $x\in B$ and a non-negative function $f$,
\begin{align*}
I_\alpha^{b,m}(f\chi_{\tilde{B}})(x)=\int_{\tilde{B}}(b(x)-b(y))^m\frac{f(y)}{|x-y|^{n-\alpha}}dy
=\varphi(x)^m\int_{\tilde{B}}\frac{f(y)}{|x-y|^{n-\alpha}}dy.
\end{align*}
By \eqref{duality}, we immediately get that
\begin{align*}
\Big(\int_B\Big(\int_{\tilde{B}}\frac{f(y)}{|x-y|^{n-\alpha}}dy\Big)^{p'}|\varphi(x)|^{mp'}
\nu(x)^{1-p'}dx\Big)^{1/p'}\lesssim\|f\chi_{\tilde{B}}/\mu\|_{L^{q',1}(\mu)}.
\end{align*}
This, combining with Lemma \ref{lm3.1}, yields that
\begin{align*}
&\frac{1}{A^{n-\alpha}}\Big(\int_B|\varphi(x)|^{mp'}\nu(x)^{1-p'}dx\Big)^{1/p'}f_{\tilde{B}}\\
&\quad=\frac{r^{n-\alpha}}{(Ar)^{n-\alpha}}\Big(\int_B|\varphi(x)|^{mp'}\nu(x)^{1-p'}dx\Big)^{1/p'}
f_{\tilde{B}}\\
&\quad=r^{-\alpha}\Big(\int_B\Big(\int_{\tilde{B}}\frac{f(y)}{|x_0-y_0|^{n-\alpha}}dy\Big)^{p'}
|\varphi(x)|^{mp'}\nu(x)^{1-p'}dx\Big)^{1/p'}\\
&\quad\leq r^{-\alpha}\Big(\int_B\Big(\int_{\tilde{B}}\Big|\frac{1}{|x_0-y_0|^{n-\alpha}}-
\frac{1}{|x-y|^{n-\alpha}}\Big|f(y)dy\Big)^{p'}|\varphi(x)|^{mp'}\nu(x)^{1-p'}dx\Big)^{1/p'}\\
&\qquad+r^{-\alpha}\Big(\int_B\Big(\int_{\tilde{B}}
\frac{1}{|x-y|^{n-\alpha}}f(y)dy\Big)^{p'}|\varphi(x)|^{mp'}\nu(x)^{1-p'}dx\Big)^{1/p'}\\
&\quad\lesssim\frac{\epsilon_A}{A^{n-\alpha}}\Big(\int_B|\varphi(x)|^{mp'}\nu(x)^{1-p'}dx\Big)^{1/p'}
f_{\tilde{B}}+r^{-\alpha}\|f\chi_{\tilde{B}}/\mu\|_{L^{q',1}(\mu)}.
\end{align*}
Choosing $A$ large enough, we have
\begin{align*}
\Big(\int_B|\varphi(x)|^{mp'}\nu(x)^{1-p'}dx\Big)^{1/p'}f_{\tilde{B}}\lesssim
r^{-\alpha}\|f\chi_{\tilde{B}}/\mu\|_{L^{q',1}(\mu)}.
\end{align*}
Taking $f=\mu$ and using the fact that $\|\chi_{\tilde{B}}\|_{L^{q',1}(\mu)}\sim(\int_{\tilde{B}}\mu)^{1/q'}$, we obtain
\begin{align}\label{3.2}
r^\alpha|\tilde{B}|^{-1}\Big(\int_B|\varphi(x)|^{mp'}\nu(x)^{1-p'}dx\Big)^{1/p'}
\Big(\int_{\tilde{B}}\mu(x)dx\Big)^{1/q}\lesssim1.
\end{align}
Similarly, let $b=\chi_B$, following the arguments as \eqref{3.2}, we get that
\begin{align}\label{3.3}
r^\alpha|\tilde{B}|^{-1}\Big(\int_B\nu(x)^{1-p'}dx\Big)^{1/p'}
\Big(\int_{\tilde{B}}\mu(x)dx\Big)^{1/q}\lesssim1.
\end{align}
Observe that $|\tilde{B}|\sim|Q|$ and $Q\subset\theta\tilde{B}$, where $\theta$ depends only on $A$ and $n$. Combing with these facts and the doubling property of $\mu$, we can replace \eqref{3.2} and \eqref{3.3} by
\begin{align*}
r^\alpha|Q|^{-1}\Big(\int_Q|g(x)-g_Q|^{mp'}\nu(x)^{1-p'}dx\Big)^{1/p'}
\Big(\int_{Q}\mu(x)dx\Big)^{1/q}\lesssim1
\end{align*}
and
\begin{align*}
r^\alpha|Q|^{-1}\Big(\int_Q\nu(x)^{1-p'}dx\Big)^{1/p'}
\Big(\int_{Q}\mu(x)dx\Big)^{1/q}\lesssim1,
\end{align*}
respectively. Keeping in mind that $g_Q\lesssim1$, we finally have
\begin{align*}
&r^\alpha|Q|^{-1}\Big(\int_Qg(x)^{mp'}\nu(x)^{1-p'}dx\Big)^{1/p'}
\Big(\int_{Q}\mu(x)dx\Big)^{1/q}\\
&\quad\lesssim r^\alpha|Q|^{-1}\Big(\int_Q|g(x)-g_Q|^{mp'}\nu(x)^{1-p'}dx\Big)^{1/p'}
\Big(\int_{Q}\mu(x)dx\Big)^{1/q}\\
&\qquad+r^\alpha|Q|^{-1}\Big(\int_Q\nu(x)^{1-p'}dx\Big)^{1/p'}
\Big(\int_{Q}\mu(x)dx\Big)^{1/q}\lesssim1,
\end{align*}
which implies that
\begin{align*}
&\sup_Q|Q|^{\frac{\alpha}{n}+\frac{1}{q}-\frac{1}{p}}
\Big(\frac{1}{|Q|}\int_Q\mu(x)dx\Big)^{1/q}\\
&\quad\times\Big(\frac{1}{|Q|}\int_Q\nu(x)^{1-p'}\Big[\log\Big(\frac{v(x)^{1-p'}}
{(v(x)^{1-p'})_Q}+e\Big)\Big]^{mp'}dx\Big)^{1/p'}<\infty.
\end{align*}
Therefore, using the following fact proved in \cite{W},
$$\|f\|_{L(\log L)^\alpha,Q}\sim\frac{1}{|Q|}\int_Q|f(x)|[\log(|f(x)|/|f|_Q+e)]^\alpha dx,$$
we get the desired result. Theorem \ref{theorem1.7} is proved.
\end{proof}

\section{Converse to Bloom type estimate for $I_\alpha^{b,m}$}

This section is concerning with the proof of Theorem \ref{theorem1.8}. First, we recall and establish some lemmas, which are the keys in our arguments.

\begin{lemma}{\rm(cf. \cite{LerOR})}\label{lm3.3}
Let $\eta_1,\eta_2$ be the weights such that $\eta_1/\eta_2\not\in L^\infty$. Then there exists $b\in BMO_{\eta_1}\backslash BMO_{\eta_2}$.
\end{lemma}

\begin{lemma}\label{lm3.4}
Let $\lambda,\mu$ be arbitrary weights satisfy \eqref{boundedness} and $p,q,m,\alpha$ be given in Theorem \ref{theorem1.8}. Then for each ball $B:=B(y_0,r)$, there exists a disjoint ball $\tilde{B}:=B(x_0,r)$ with dist$(B,\tilde{B})\sim Ar$ such that for any non-negative measurable function $f$,
\begin{align*}
\Big(\int_{\tilde{B}}\eta(x)^{mq}\lambda(x)^qdx\Big)^{1/q}f_B\lesssim r^{-\alpha}\Big(\int_Bf(x)^p\mu(x)^pdx\Big)^{1/p}.
\end{align*}
\end{lemma}
\begin{proof}
Let $\tilde{B}$ be given in Lemma \ref{lm3.1} and $b=\eta\chi_{\tilde{B}}$. Then for $x\in\tilde{B}$,
\begin{align*}
I_{\alpha}^{b,m}(f\chi_B)(x)=\int_B(b(x)-b(y))^m\frac{f(y)}{|x-y|^{n-\alpha}}dy=\eta(x)^m
\int_B\frac{f(y)}{|x-y|^{n-\alpha}}dy.
\end{align*}
By \eqref{boundedness}, we have
\begin{align*}
\Big(\int_{\tilde{B}}\Big(\int_B\frac{f(y)}{|x-y|^{n-\alpha}}dy\Big)^q
\eta(x)^{mq}\lambda(x)^qdx\Big)^{1/q}
\lesssim\Big(\int_Bf(x)^p\mu(x)^pdx\Big)^{1/p}.
\end{align*}
From this and Lemma \ref{lm3.1}, we deduce that
\begin{align*}
&\frac{1}{A^{n-\alpha}}\Big(\int_{\tilde{B}}\eta(x)^{mq}\lambda(x)^qdx\Big)^{1/q}f_B\\
&\quad=\frac{r^{n-\alpha}}{(Ar)^{n-\alpha}}\Big[\int_{\tilde{B}}\eta(x)^{mq}\lambda(x)^q
\Big(\frac{1}{|B|}\int_Bf(y)dy\Big)^qdx\Big]^{1/q}\\
&\quad=r^{-\alpha}\Big[\int_{\tilde{B}}\eta(x)^{mq}\lambda(x)^q
\Big(\int_B\frac{f(y)}{|x_0-y_0|^{n-\alpha}}dy\Big)^qdx\Big]^{1/q}\\
&\quad\leq r^{-\alpha}\Big[\int_{\tilde{B}}\eta(x)^{mq}\lambda(x)^q
\Big(\int_B\Big|\frac{1}{|x-y|^{n-\alpha}}-\frac{1}{|x_0-y_0|^{n-\alpha}}\Big|f(y)dy\Big)^qdx\Big]
^{1/q}\\
&\qquad+r^{-\alpha}\Big[\int_{\tilde{B}}\eta(x)^{mq}\lambda(x)^q
\Big(\int_B\frac{1}{|x-y|^{n-\alpha}}f(y)dy\Big)^qdx\Big]^{1/q}\\
&\quad\lesssim\frac{\epsilon_A}{A^{n-\alpha}}\Big(\int_{\tilde{B}}\eta(x)^{mq}\lambda(x)^qdx\Big)
^{1/q}f_B+r^{-\alpha}\Big(\int_Bf(x)^p\mu(x)^pdx\Big)^{1/p}.
\end{align*}
Then the desired result directly follows by letting $A\rightarrow\infty$.
\end{proof}

\begin{lemma}\label{lm3.5}
Let $\lambda,\mu$ be arbitrary weights satisfy \eqref{boundedness} and $p,q,m,\alpha$ be given in Theorem \ref{theorem1.8}. Then
\begin{align*}
\lambda(x)\eta(x)^{m}\lesssim\mu(x).
\end{align*}
\end{lemma}
\begin{proof}
Let $f=1$ in Lemma \ref{lm3.4}. Keep in mind that $1/p-1/q=\alpha/n$, then
\begin{align*}
\Big(\frac{1}{|\tilde{B}|}\int_{\tilde{B}}\eta(x)^{mq}\lambda(x)^qdx\Big)^{1/q}\lesssim
\Big(\frac{1}{|B|}\int_{B}\mu(x)^pdx\Big)^{1/p}.
\end{align*}
By the Lebesgue differential theorem, we get the desired result.
\end{proof}

Now, we are in the position to prove Theorem \ref{theorem1.8}.

\begin{proof}[Proof of Theorem \ref{theorem1.8}]
By Lemma \ref{lm3.5}, it suffices to prove that
\begin{align}\label{3.4}
\mu\lesssim\lambda\eta^m
\end{align}
almost everywhere. Suppose that \eqref{3.4} is not true. Denote $\tilde{\eta}=(\mu/\lambda)^{1/m}$. Then $\tilde{\eta}/\eta\not\in L^\infty$. Note that when  $\lambda,\,\mu\in A_{p,q}$, Accomazzo et al. \cite{AMR} proved that for $b\in BMO_{\tilde{\eta}}$,
\begin{align*}
\|I_{\alpha}^{b,m}f\|_{L^q(\lambda^q)}\lesssim\|f\|_{L^p(\mu^p)}.
\end{align*}
This, together with Lemma \ref{lm3.3}, implies that $b\not\in BMO_\eta$, which contradicts with \eqref{necessity} and completes the proof of Theorem \ref{theorem1.8}.
\end{proof}

%\subsection*{Acknowledgements} The authors thank the referee cordially for their valuable suggestions.
%%%%%%%%%%%% References %%%%%%%%%%%%%
%%
%<Author name> is written as Initial of Given Name, and Family Name.
%<Title> is written in roman letters.
%<Journal name> should be abbreviated according to
% the MR Serials Abbreviations List of Mathematical Reviews:
% (Abbreviations of Names of Serials; http://www.ams.org/mr-database)
%For <Pages>, use en-dash "--" between page numbers.
%%


\begin{thebibliography}{99}

%%% Journals %%%

%<Author name>, <Title>, <Journal name> <Volume> (<Year>), <Pages>.
\bibitem{AMR}N. Accomazzo, J.C. Mart\'{i}nez-Perales and I.P. Rivera-R\'{i}os,
On Bloom type estimates for iterated commutators of fractional integrals,
Indiana Univ. Math. J. 69(4) (2020), 1207--1230.

\bibitem{Bl0}S. Bloom,
A commutator theorem and weighted BMO,
Trans. Amer. Math. Soc. 292(1) (1985), 103--122.

\bibitem{CaIs}R. Cardenas and J. Isralowitz,
Two matrix weighted inequalities for commutators with fractional integral operators,
arXiv: 2101.12082v1.

\bibitem{Chan}S. Chanillo,
A note on commutators,
Indiana Univ. Math. J. 31(1) (1982), 7--16.

\bibitem{Cru}D. Cruz-Uribe,
Two weight inequalities for fractional integral operators and commutators,
Advanced courses of mathematical analysis VI, World Sci. Publ., Hackensack, NJ, (2017), 25--85.

\bibitem{CruMP}D. Cruz-Uribe, J.M. Martell and C. P\'{e}rez,
Weights, extrapolation and the theory of Rubio de Francia. Operator Theory: Advances and Applications,
Birkh\"{a}user/Springer Basel AG, Basel, 215 (2011), xiv+280 pp. ISBN: 978--3--0348--0071--6.

\bibitem{CruzM0}D. Cruz-Uribe and K. Moen,
Sharp norm inequalities for commutators of classical operators,
Publ. Mat. 56(1) (2012), 147--190.

\bibitem{CruzM}D. Cruz-Uribe and K. Moen,
A fractional Muckenhoupt-Wheeden theorem and its consequences,
Integral Equations Operator Theory 76(3) (2013), 421--446.

\bibitem{CruMT}D. Cruz-Uribe, K. Moen and Q.M. Tran,
New oscillation classes and two weight bump conditions for commutators,
arXiv: 2103.06821v1.

\bibitem{CruPe00}D. Cruz-Uribe and C. P\'{e}rez,
Sharp two-weight, weak-type norm inequalities for singular integral operators,
Math. Res. Lett. 6(3-4) (1999), 417--427.

\bibitem{CruPe0}D. Cruz-Uribe and C. P\'{e}rez,
Two-weight, weak-type norm inequalities for fractional integrals, Calder\'{o}n-Zygmund
operators and commutators,
Indiana Univ. Math. J. 49(2) (2000), 697--721.

\bibitem{CruPe}D. Cruz-Uribe and C. P\'{e}rez,
On the two-weight problem for singular integral operators,
Ann. Sc. Norm. Super. Pisa Cl. Sci. (5) 1(4) (2002), 821--849.

\bibitem{CruRV}D. Cruz-Uribe, A. Reznikov and A. Volberg,
Logarithmic bump conditions and the two-weight boundedness of Calder\'{o}n-Zygmund operators,
Adv. Math. 255 (2014), 706--729.

\bibitem{HRS}I. Holmes, R. Rahm and S. Spencer,
Commutators with fractional integral operators.
Studia Math. 233(1) (2016), 279--291.

\bibitem{IsPT}J. Isralowitz, S. Pott and S. Treil,
Commutators in the two scalar and matrix weighted setting,
arXiv: 2001.11182v1.

\bibitem{Lacey}M.T. Lacey,
On the separated bumps conjecture for Calde\'{o}n-Zygmund operators,
Hokkaido Math. J. 45(2) (2016), 223--242.

\bibitem{Lerner}A.K. Lerner,
On an estimate of Calder\'{o}n-Zygmund operators by dyadic positive operators,
J. Anal. Math. 121 (2013), 141--161.

\bibitem{LerNa}A.K. Lerner and F. Nazarov,
Intuitive dyadic calculus: the basics,
Expo. Math. 37(3) (2019), 225--265.

\bibitem{LerOR}A.K. Lerner, S. Ombrosi and I.P. Rivera-R\'{i}os,
On two weight estimates for iterated commutators,
arXiv: 2006.11896v1.

\bibitem{Li}K. Li,
Two weight inequalities for bilinear forms,
Collect. Math. 68(1) (2017), 129--144.

\bibitem{MuW}B. Muckenhoupt and R. Wheeden,
Weighted norm inequalities for fractional integrals,
Trans. Amer. Math. Soc. 192 (1974), 261--274.

\bibitem{NaReTV}F. Nazarov, A. Reznikov, S. Treil and A. Volberg,
A Bellman function proof of the $L^2$ bump conjcture,
J. Anal. Math. 121 (2013), 255--277.

\bibitem{Neu}C.J. Neugebauer,
Inserting $A_p$-weights,
Proc. Amer. Math. Soc. 87(4) (1983), 644--648.

\bibitem{Perey}M.C. Pereyra,
Dyadic harmonic analysis and weighted inequalities: the sparse revolution,
New Trends in Applied Harmonic Analysis, Volume 2. Birkh\"{a}user, Cham (2019): 159--239.

\bibitem{Perez2}C. P\'{e}rez,
Two weighted inequalities for potential and fractional type maximal operators,
Indiana Univ. Math. J. 43(2) (1994), 663--683.

\bibitem{Perez3}C. P\'{e}rez,
On sufficient conditions for the boundedness of the Hardy-Littlewood maximal operator between weighted $L^p$-spaces with different weights,
Proc. London Math. Soc. (3) 71(1) (1995), 135--157.

\bibitem{R}R. Rahm,
Off-diagonal two weight bumps for fractional sparse operators,
arXiv: 2101. 02123v2.

\bibitem{Sawyer}E.T. Sawyer,
A characterization of a two-weight norm inequality for maximal operators,
Studia Math. 75(1) (1982), 1--11.

\bibitem{ST}C. Segovia and J.L. Torrea,
Weighted inequalities for commutators of fractional and singular integrals, Conference on Mathematical Analysis (EI Escorial, 1989). Publ. Mat. 35(1) (1991), 209--235.

\bibitem{W}M. Wilson,
Weighted Littlewood-Paley theory and exponential-square integrability,
Lecture Notes in Mathematics, 1924, Springer, Berlin, (2008), xiv+224 pp. ISBN:978--3--540--74582--2.
\end{thebibliography}
\end{document}